\documentclass[a4paper,12pt]{article}
\usepackage{amssymb}
\usepackage{latexsym}
\usepackage{amsmath,amsthm,amssymb}
\usepackage{amsfonts}
\usepackage{eufrak}

\newtheorem{thm}{Theorem}[section]
\newtheorem{lem}[thm]{Lemma}
\newtheorem{pro}[thm]{Proposition}

\theoremstyle{definition}

\newtheorem{exa}[thm]{Example}
\newtheorem{rem}[thm]{Remark}

\begin{document}

\begin{center}
{\Large On the strong law of large numbers for $\varphi$-subgaussian random variables}% satisfying the Cram\'er condition}
\end{center}
\begin{center}
{\sc Krzysztof Zajkowski}\\
%\footnote{The author is supported by the Polish National Science Center, Grant no. DEC-2011/01/B/ST1/03838.}\\
Institute of Mathematics, University of Bialystok \\ 
Ciolkowskiego 1M, 15-245 Bialystok, Poland \\ 
kryza@math.uwb.edu.pl 
\end{center}

\begin{abstract}
For $p\ge 1$ let $\varphi_p(x)=x^2/2$ if $|x|\le 1$ and $\varphi_p(x)=1/p|x|^p-1/p+1/2$ if $|x|>1$. %($x\in\mathbb{R}$). 
For a random variable $\xi$ let $\tau_{\varphi_p}(\xi)$ denote
$\inf\{a\ge 0:\;\forall_{\lambda\in\mathbb{R}}\; \ln\mathbb{E}\exp(\lambda\xi)\le\varphi_p(a\lambda)\}$; $\tau_{\varphi_p}$ is a norm in a space $Sub_{\varphi_p}=\{\xi:\;\tau_{\varphi_p}(\xi)<\infty\}$ of $\varphi_p$-subgaussian random variables. We prove that if for a sequence $(\xi_n)\subset Sub_{\varphi_p}$ ($p>1$) there exist positive constants $c$ and $\alpha$ such that for every natural number $n$ 
the following inequality $\tau_{\varphi_p}(\sum_{i=1}^n\xi_i)\le cn^{1-\alpha}$ holds then  $n^{-1}\sum_{i=1}^n\xi_i$ converges almost surely to zero  as $n\to\infty$. This result is a generalization of the SLLN for independent subgaussian random variables (Taylor and Hu \cite{TayHu}) to the case of dependent $\varphi_p$-subgaussian random variables.
% $1/n^r\sum_{i=1}^n\xi_i\rightarrow 0$ almost surely.
\end{abstract}



\section{Introduction}
The classical Kolmogorov strong laws of large numbers are dealt with independent variables.  Investigations of limit theorems for dependent r.v.s are extensive and episodic. The strong law of large numbers for various classes of many type associated random variables one can find for instance in Bulinski and Shashkin \cite[Chap. 4]{BuSh}. Most of them are considered in the spaces of integrable functions. It is also interested  to describe  general conditions under which the SLLN holds in other spaces of random variables than $L_p$-spaces. In this paper we investigate almost sure convergence of the arithmetic mean (but not only) sequences of $\varphi$-subgaussian random variables.

The notion of subgaussian random variables was introduced by Kahane in \cite{Kahane}. A random variable $\xi$ is {\it subgaussian} %($\sigma$-{\it subgaussian}) 
if its moment generating function is majorized by the moment generating function of some centered Gaussian r.v. with variance $\sigma^2$ that is  $\mathbb{E}\exp(\lambda \xi)\leq \mathbb{E}\exp(\lambda g)=\exp(\sigma^2\lambda^2/2)$, where $g\sim\mathcal{N}(0,\sigma^2)$ (see Buldygin and Kozachenko \cite{BK} or \cite[Ch.1]{BulKoz}). In terms of the cumulant generating functions this condition takes a form:
$\ln\mathbb{E}\exp(\lambda \xi)\leq \sigma^2\lambda^2/2$.  

One can generalize the notion of subgaussian r.v.s to  classes of $\varphi$-subgaussian random variables (see \cite[Ch.2]{BulKoz}). A continuous even convex function $\varphi(x)$ $(x\in \mathbb{R}$) is called a {\em $N$-function}, if the following condition hold:\\
(a) $\varphi(0)=0$ and $\varphi(x)$ is monotone increasing for $x>0$,\\
(b) $\lim_{x\to 0}\varphi(x)/x=0$ and $\lim_{x\to \infty}\varphi(x)/x=\infty$.\\
It is called a {\it quadratic $N$-function}, if in addition $\varphi(x)=cx^2$ for all $|x|\le x_0$, with $c>0$ and $x_0>0$. The  quadratic condition is needed to ensure nontriviality for classes of $\varphi$-subgaussian random variables (see \cite[Ch.2, p.67]{BulKoz}).

\begin{exa}
Let for $p\ge 1$
$$
\varphi_p(x)=\left\{
\begin{array}{ccl}
\frac{x^2}{2}, & {\rm if} & |x|\le 1,\\
\frac{1}{p}|x|^p-\frac{1}{p}+\frac{1}{2}, & {\rm if} & |x|>1.
\end{array}
\right.
$$
The function $\varphi_p$ is an example of the quadratic $N$-function which is some standardization of the function $|x|^p$ (see \cite[Lem. 2.5]{Zaj1}).
Let us emphasize that for $p=2$ we have the case of subgaussian random variables.
\end{exa}

Let $\varphi$ be a quadratic $N$-function. A random variable $\xi$ is said to be {\it $\varphi$-subgaussian} if there is a constant $a>0$ such that 
$\ln\mathbb{E}\exp(\lambda \xi)\leq \varphi(a\lambda)$. The {\it $\varphi$-subgaussian standard (norm) $\tau_{\varphi}(\xi)$} is defined as 
$$
\tau_{\varphi}(\xi)=\inf\{a\ge 0:\;\forall_{\lambda\in\mathbb{R}}\; \ln\mathbb{E}\exp(\lambda \xi)\le\varphi(a\lambda)\};
$$
a space $Sub_\varphi=\{\xi:\;\tau_{\varphi}(\xi)<\infty\}$ with the norm $\tau_{\varphi}$ is a Banach space (see \cite[Ch.2, Th.4.1]{BulKoz})

Let $\varphi(x)$ ($x\in\mathbb{R}$) be a real-valued function. The function $\varphi^\ast(y)$ ($y\in\mathbb{R}$) defined by $\varphi^\ast(y)=\sup_{x\in\mathbb{R}}\{xy-\varphi(x)\}$ is called the {\it Young-Fenchel transform}
or the {\it convex conjugate} of $\varphi$ (in general, $\varphi^\ast$ may take value $\infty$). It is known that if $\varphi$ is a quadratic $N$-function then $\varphi^\ast$ is quadratic $N$-function too. For instance, since our $\varphi_p$ is a differentiable (even at $\pm 1$) function one can easy check (see \cite[Lem. 2.6]{Zaj1}) that $\varphi_p^\ast=\varphi_q$ for $p,q>1$, if $1/p+1/q=1$.
\begin{rem}
One can define the space $Sub_{\varphi_p}$ by using the Luxemburg norm of the form
$$
\|\xi\|_{\psi_q}=\inf\big\{K>0:\;\mathbb{E}\exp|\xi/K|^q\le 2\big\}\quad (q=p/(p-1)),
$$
and then $Sub_{\varphi_p}=\{\xi:\;\|\xi\|_{\psi_q}<\infty\;{\rm and}  \;\mathbb{E}\xi=0\}$ (compare \cite[Th. 2.7]{Zaj1}, the space $L_{\psi_q}^0=Sub_{\varphi_p}$). Let us note that 
$\|\mathbb{E}\xi\|_{\psi_q}=\|\xi\|_{\psi_q}$  and we get that if $\|\xi\|_{\psi_q}<\infty$ then $\xi-\mathbb{E}\xi\in Sub_{\varphi_p}$.
\end{rem}
\begin{exa}
The standard normal random variable $g$ belongs to $Sub_{\varphi_2}$ and $\tau_{\varphi_2}(g)=1$, since
$$
\mathbb{E}\exp(tg)=\exp(t^2/2)=\exp(\varphi_2(t)).
$$
Because $g^2$ has $\chi^2_1$-distribution with one degree of freedom whose moment generating function is $\mathbb{E}\exp(tg)=(1-2t)^{-1/2}$ for $t<1/2$ then
$$
\mathbb{E}\exp(g^2/K^2)=(1-2/K^2)^{-1/2},
$$
which is less or equal $2$ if $K\ge \sqrt{8/3}$. It gives that $\|\xi\|_{\psi_2}=\sqrt{8/3}$. Let us observe that $\psi_2$-norm of $|g|$ is equal to $\psi_2$-norm of $g$. It implies that $|g|-\mathbb{E}|g|\in Sub_{\varphi_2}$.
Similarly as above one can show that
$$
\||g|^{2/q}\|_{\psi_q}=\inf\big\{K>0;\;\mathbb{E}\exp(g^2/K^q)\le 2\big\}=(8/3)^{1/q}<\infty.
$$
Thus we get that $|g|^{2/q}-\mathbb{E}|g|^{2/q}\in Sub_{\varphi_p}$, where $1/p+1/q=1$.

%Let $g\sim\mathcal{N}(0,1)$ then $\xi=|g|^{2/q}-\mathbb{E}|g|^{2/q}\in Sub_{\varphi_p}(\Omega)$, where $1/p+1/q=1$.
\end{exa}

Let us recall that the convex conjugate is order-reversing and possesses some scaling property. If $\varphi_1\ge\varphi_2$ then $\varphi_1^\ast\le\varphi_2^\ast$.
Let for $a>0$ and $b\neq 0$ $\psi(x)=a\varphi(bx)$ then $\psi^\ast(y)=a\varphi^\ast(y/(ab))$ (see e.g.\cite[Ch.X, Prop.1.3.1]{HUiL}).

The convex conjugate of the cumulant generating function can be served to estimate of 'tails' distribution of a centered random variable. Let $\mathbb{E}\xi=0$ and $\psi_\xi$ denote the cumulant generating function of $\xi$, i.e. $\psi_\xi(\lambda)=\ln\mathbb{E}\exp(\lambda\xi)$ then for $\varepsilon>0$
$$
\mathbb{P}(\xi\ge\varepsilon)\le\exp(-\psi_\xi^\ast(\varepsilon)).
$$

Let us observe that for $\xi\in Sub_\varphi$, by the definition of $\tau_{\varphi}(\xi)$, we have the following inequality: 
$\psi_\xi(\lambda)\le\varphi(\tau_{\varphi}(\xi)\lambda)$ and by the order-reversing and the scaling property we get    
$\psi_\xi^\ast(\varepsilon)\ge\varphi^\ast(\varepsilon/\tau_{\varphi}(\xi))$. Now we can obtain some weaker form of the above estimation but with using the general function $\varphi$:
\begin{equation}
\label{estm}
\mathbb{P}(|\xi|\ge\varepsilon)\le 2\exp\Big(-\varphi^\ast\Big(\frac{\varepsilon}{\tau_{\varphi}(\xi)}\Big)\Big);
\end{equation}
see \cite[Ch.2, Lem.4.3]{BulKoz}.  

%For the sake of completennes we racall a proof of some so-called scaling property of convex conjugates.
%\begin{lem}
%Let $\varphi$ be a function on $\mathbb{R}$ and let for $a>0$ and $b\neq 0$ $\psi(x):=a\varphi(bx)$ then $\psi^\ast(y)=a\varphi^\ast(y/(ab))$.
%\end{lem}
%\begin{proof}
%\begin{eqnarray*}
%\psi^\ast(y) & = & \sup_{x\in\mathbb{R}}\{xy-a\varphi(bx)\}=a\sup_{x\in\mathbb{R}}\{bxy/(ab)-\varphi(bx)\}\\
%\; & = & a\sup_{t\in\mathbb{R}}\{ty/(ab)-\varphi(t)\}\quad {\rm where}\quad t=bx\\
%\; & = & a\varphi^\ast(y/(ab))
%\end{eqnarray*}

%\end{proof}

%In Theorem \ref{mtw} some sufficient condition  for the applicability of the strong law of large numbers to a sequence of dependent  subgassian random variables of rank $p$ is formulated.
%We formulate some sufficient condition to achive for a sequence of subgaussian  random variables of rank $p$ the strong law of large numbers. 

\section{Results}
First we show that if we have some upper estimate for $\tau_\varphi$ then in  (\ref{estm}) we can substitute this estimate instead of $\tau_{\varphi}$.
%We begin with another form of an estimate  
\begin{lem} 
\label{estm1}
If $\tau_{\varphi}(\xi)\le C(\xi)$ for every $\xi\in Sub_\varphi$ then
\begin{equation*}
\mathbb{P}(|\xi|\ge\varepsilon)\le 2\exp\Big(-\varphi^\ast\Big(\frac{\varepsilon}{C(\xi)}\Big)\Big).
\end{equation*}
\end{lem}
\begin{proof}
Since $\varphi$ is even and increasing monotonic for $x>0$, we get
$$
\varphi(\tau_{\varphi}(\xi)x)=\varphi(\tau_{\varphi}(\xi)|x|)\le\varphi(C(\xi)|x|)
=\varphi(C(\xi)x).
$$
And again by the order-reversing and the scaling property we obtain
$$
\varphi^\ast\Big(\frac{y}{\tau_{\varphi}(\xi)}\Big)\ge \varphi^\ast\Big(\frac{y}{C(\xi)}\Big),
$$
which combined with (\ref{estm}) establishes the inequality.
\end{proof}
With these preliminaries accounted for, we can  prove the main result of the paper.
\begin{thm}
\label{mthm}
Let  $(\xi_n)\subset Sub_{\varphi_p}$ for some  $p>1$. If there exist positive constants $c$ and $\alpha$ such that for every natural number $n$ 
the following condition $\tau_{\varphi_p}(\sum_{i=1}^n\xi_i)\le cn^{1-\alpha}$ holds then the term $n^{-1}\sum_{i=1}^n\xi_i$ converges almost surely to zero  as $n\to\infty$ .
\end{thm}
\begin{proof}
Since $\varphi_p^\ast=\varphi_q$, by Lemma \ref{estm1} and  the condition of the theorem we have 
$$
\mathbb{P}\Big(\Big|\sum_{i=1}^n\xi_i\Big|\ge n\varepsilon\Big)%\le 2\exp\Big(-\varphi_q\Big(\frac{n\varepsilon}{\tau_{\varphi_p}(\sum_{i=1}^n\xi_i)}\Big)\Big)
\le 2\exp\Big(-\varphi_q\Big(\frac{n^\alpha\varepsilon}{c}\Big)\Big).
$$
For sufficiently large $n$ ($n>(c/\varepsilon)^{1/\alpha}$) we have $n^\alpha\varepsilon/c>1$ and, in consequence,  
%$$
%\frac{n\varepsilon}{\tau_{\varphi_p}(\sum_{i=1}^n\xi_i)}\ge \frac{n\varepsilon}{cn^{1-\alpha}}=\frac{n^\alpha\varepsilon}{c}\ge 1
%$$
%and so
$$
\varphi_q\Big(\frac{n^\alpha\varepsilon}{c}\Big)=n^{q\alpha}\frac{1}{q}\Big(\frac{\varepsilon}{c}\Big)^q-\frac{1}{q}+\frac{1}{2}. 
$$
%For $0<r\le q$ we have $n^{q/r}\ge n$
Thus we get the following estimate 
$$
\mathbb{P}\Big(\Big|\sum_{i=1}^n\xi_i\Big|\ge n\varepsilon\Big)%\le 2\exp\Big(-n^\frac{q}{r}\frac{1}{q}\Big(\frac{\varepsilon}{b}\Big)^q\Big)
\le 2\exp\Big(\frac{1}{q}-\frac{1}{2}\Big)\exp\Big(-n^{q\alpha}\frac{1}{q}\Big(\frac{\varepsilon}{c}\Big)^q\Big) 
$$
for every $\varepsilon$ and $n>(c/\varepsilon)^{1/\alpha}$.
Thus, by the integral test, we obtain convergence of the series $\sum_{n=1}^\infty\mathbb{P}(|\sum_{i=1}^n\xi_i|\ge n\varepsilon)$.
It follows the completely and, in consequence,  almost sure convergence  of $n^{-1}\sum_{i=1}^n\xi_i$ to zero.

\end{proof}
\begin{rem}
Let us emphasize that the above theorem is a generalization of the Theorem (SLLN) (Taylor and Hu \cite[sec.3, p.297]{TayHu}) to the case of $\varphi_p$-subgaussian random variables, not only subgaussian ones. Moreover we do not assume their independence. For this reason we used a %slightly 
modified condition for a behavior of the norm $\tau_p$ than Taylor and Hu, which I describe below.
\end{rem}

Since $\tau_\varphi$ is a norm, we obtain 
$$
\tau_{\varphi}\Big(\sum_{i=1}^n\xi_i\Big)\le \sum_{i=1}^n\tau_{\varphi}(\xi_i).
$$
If for instance $\xi_i$, $i=1,...,n$, are copies of the same variable $\xi$ then in the above the equality holds and $\tau_{\varphi}\Big(\sum_{i=1}^n\xi_i\Big)=n\tau_\varphi(\xi)$. Let us observe that in this case the assumption of Theorem \ref{mthm} is not satisfied. Additionally informations about form of dependence (or independence) sometime allow us to improve this estimate. And so, for an independence sequence 
$(\xi_n)$ if there is some $r\in(0,2]$ such that $\varphi(|x|^{1/r})$ is convex then 
\begin{equation}
\label{enorm}
\tau_{\varphi}\Big(\sum_{i=1}^n\xi_i\Big)^r\le \sum_{i=1}^n\tau_{\varphi}(\xi_i)^r;
\end{equation}
see \cite[Sec.2, Th.5.2]{BulKoz}. If $r$ is bigger  then the estimate is better. For the function $\varphi_p$ we can always take $r=\min\{p,2\}$.
In Taylor's and Hu's SLLN  variables $\xi_n$ were subgaussian and independent and it was taken $p=2$.  Let us emphasize that in this case if in addition
$\xi_1,...,\xi_n$ have the same distribution as $\xi$ then $\tau_{\varphi}(\sum_{i=1}^n\xi_i) \le \sqrt{n}\tau_\varphi(\xi)$ and the condition of Theorem \ref{mthm} is satisfied ($c=\tau_\varphi(\xi)$ and $\alpha=1/2$). 

Let us emphasize that another assumptions on dependence of $\xi_1,...,\xi_n$ can give the same estimate of the norm of  $\tau_{\varphi}(\sum_{i=1}^n\xi_i)$.
In the paper Giuliano Antonini et al.\cite[Lem.3]{Rita} it was proved that for $\varphi$-subgaussian  acceptable random variables the inequality (\ref{enorm}) holds,
if $\varphi(|x|^{1/r})$ is convex. The definition of acceptability of sequence of random variable one can find  therein. For us it is the most important that these estimates are the same.  %we have the same estimation of $\tau_\varphi(\sum_{i=1}^n\xi_i)$.
In this article there is some version of the Marcinkiewicz-Zygmund law of large numbers for $\varphi$-subgaussian random variables as a corollary of much more general theorem. We give an independent proof of this corollary but under modified assumptions.
\begin{pro}
\label{MZ}
For $p>1$ let $(\xi_n)$ be a bounded sequence of $\varphi_p$-subgaussian random variables and let $r=\min\{p,2\}$. If in addition
\begin{equation}
\label{estm2}
\tau_{\varphi_p}\Big(\sum_{i=1}^n\xi_i\Big)^r\le \sum_{i=1}^n\tau_{\varphi_p}(\xi_i)^r
\end{equation}
then
$
n^{-1/s}\sum_{i=1}^n\xi_i \rightarrow 0 
$
almost surely for any $0<s<r$.
\end{pro} 
\begin{rem}
Since $\varphi_p(|x|^{1/r})$ is convex, the estimate (\ref{estm2}) is satisfied by sequences of independent or acceptable random variables, for instance.
\end{rem}
\begin{proof}
Let $b=\sup_{n\ge 1}\tau_{\varphi_p}(\xi_n)$ then $\sum_{i=1}^n\tau_{\varphi_p}(\xi_i)^r\le nb^r$ and, in consequence,   
$\tau_{\varphi_p}\Big(\sum_{i=1}^n\xi_i\Big)\le n^{1/r}b$. For positive number $s$ less than $r$,  by Lemma \ref{estm1}, we obtain
$$
\mathbb{P}\Big(\Big|\sum_{i=1}^n\xi_i\Big|\ge n^\frac{1}{s}\varepsilon)\le 2\exp\Big(-\varphi_q\Big(\frac{n^{1/s}\varepsilon}{n^{1/r}b}\Big)\Big)
=2\exp\Big(-\varphi_q\Big(n^{(\frac{1}{s}-\frac{1}{r})}\frac{\varepsilon}{b}\Big)\Big).
$$
For $n>(b/\varepsilon)^{(1/s-1/r)^{-1}}$ we have 
$$
\varphi_q\Big(n^{(\frac{1}{s}-\frac{1}{r})}\frac{\varepsilon}{b}\Big)=n^{q(\frac{1}{s}-\frac{1}{r})}\frac{1}{q}\Big(\frac{\varepsilon}{b}\Big)^q-\frac{1}{q}+\frac{1}{2}
$$
and, in consequence, we get
$$
\sum_{n=1}^\infty\exp\Big(-\varphi_q\Big(n^{(\frac{1}{s}-\frac{1}{r})}\frac{\varepsilon}{b}\Big)\Big)<\infty,
$$
which, in view of Borel-Cantelli lemma, completes the proof.
\end{proof}
\begin{rem}
Because we apply the function $\varphi_p(x)$ instead of $|x|^p$ then we must not restrict $p$ to be less or equal $2$ to ensure the fulfillment of the quadratic condition for the  function $|x|^p$. Moreover we use the metric property (\ref{estm2}) instead of assumptions on some form of dependence random variables (compare \cite[Cor. 7]{Rita}).  
\end{rem}

\begin{exa}
The proof of Hoeffding-Azuma's inequality for a sequence $(\xi_n)$ of bounded random variables such that $|\xi_n|\le d_n$ a.s. and $\mathbb{E}\xi_n=0$ is based on an estimate of the moment generating function of the partial sum $\sum_{i=1}^n\xi_i$. Under assumptions  that $\xi_n$ are independent (Hoeffding) or $\xi_n$ are martingales increments (Azuma) the following inequality holds 
\begin{equation}
\label{HA}
\mathbb{E}\exp\Big(\lambda\sum_{i=1}^n\xi_i\Big)\le \exp\Big(\frac{\lambda^2\sum_{i=1}^nd_i^2}{2}\Big);
\end{equation}
see Hoeffding \cite{Hoeff} and Azuma \cite{Azuma}. Let us emphasize that in \cite{Azuma} Azuma  has proved the above estimate under more general assumptions on $(\xi_n)$ which  satisfy centered bounded martingales increments. The inequality (\ref{HA}) means that 
$$
\tau_{\varphi_2}\Big(\sum_{i=1}^n\xi_i\Big)\le \Big(\sum_{i=1}^nd_i^2\Big)^{1/2}.
$$
If we take $d_n=1$ for $n=1,2,...$ then we get the following condition
$$
\tau_{\varphi_2}\Big(\sum_{i=1}^n\xi_i\Big)\le \sqrt{n},
$$
which follows that the sequence $(\xi_n)$ satisfies the assumptions of Proposition \ref{MZ} with $p=r=2$ and the norm $\tau_{\varphi_2}(\xi_n)\le 1$ and we get
the almost sure convergence 
$
n^{-1/s}\sum_{i=1}^n\xi_i 
$
to $0$
for any $0<s<2$. Let us note that for $s=1$ we obtain SLLN for this sequence.
\end{exa}

\end{document}